\documentclass[12pt]{article}
\usepackage{paralist,amsmath,amsthm,amssymb,mathtools,url,graphicx,pdfpages,tikz,pdfpages,rotating}
\usepackage[affil-it]{authblk}
\usepackage[left=1in,right=1in,top=1.2in, bottom=1in]{geometry}
\usepackage{mathdots}
\usepackage{bm} 

\newtheorem{theorem}{Theorem}[section]
\newtheorem{lemma}[theorem]{Lemma}
\newtheorem{corollary}[theorem]{Corollary}

\theoremstyle{definition}
\newtheorem{definition}[theorem]{Definition}
\newtheorem{example}[theorem]{Example}
\newtheorem{obs}[theorem]{Observation}
\newtheorem{remark}[theorem]{Remark}

\newtheorem{question}[theorem]{Question}

\newtheorem{conjecture}[theorem]{Conjecture}

\numberwithin{equation}{section}

\DeclareMathOperator{\n}{N }

\DeclareMathOperator{\von}{von}

\newcommand{\Mod}[1]{ \left(\mathrm{mod}\ #1\right)}
\newcommand{\blank}{$\;\;$}
\newcommand{\F}{\mathbb{F}}

\begin{document}
\title{Isodual and Self-dual Codes from Graphs}

\author{Sudipta Mallik\thanks{Corresponding author} } 
\author{Bahattin Yildiz } 
\affil{\small Department of Mathematics and Statistics, Northern Arizona University, 801 S. Osborne Dr.\\ PO Box: 5717, Flagstaff, AZ 86011, USA 

sudipta.mallik@nau.edu, bahattin.yildiz@nau.edu}

\maketitle
\begin{abstract}
Binary linear codes are constructed from graphs, in particular, by the generator matrix $[I_n | A]$ where $A$ is the adjacency matrix of a graph on $n$ vertices. A combinatorial interpretation of the minimum distance of such codes is given. We also present graph theoretic conditions for such linear codes to be Type I and Type II self-dual. Several examples of binary linear codes produced by well-known graph classes are given.
\end{abstract}

\renewcommand{\thefootnote}{\fnsymbol{footnote}} 
\footnotetext{\emph{2010 Mathematics Subject Classification: 94B05, 94B25}}
\renewcommand{\thefootnote}{\arabic{footnote}} 

\section{Introduction}
There is a strong connection between graphs and codes. The adjacency matrix of a simple graph is a symmetric binary matrix which has made it suitable for constructing binary codes. Depending on the structure of the graphs, special types of codes can be obtained.

We start with some basic definitions about codes that  will be used throughout the paper. Let $\F_2$ be the binary field. A binary linear code $C$ of length $n$ is defined as a subspace of $\F_2^n$. If the dimension of $C$ is $k$, we say $C$ is an $[n,k]$-code. A matrix whose rows form a basis for $C$ is called a {\it generator matrix} for $C$ and is denoted by $G$. We also denote $C$ by $C(G)$. By using elementary row and column operations, we can bring the generator matrix $G$ into a standard form $[I_k|A]$ where $A$ is a $k\times (n-k)$ matrix. Two binary codes are said to be {\it equivalent} if one can be obtained from the other by a permutation of coordinates. 

The {\it Hamming weight} $w_H(\bm x)$ of a vector $\bm x \in \F_2^n$ is defined as the number of non-zero coordinates in $\bm x$. The {\it  Hamming distance} between two vectors $\bm x$ and $\bm y$ in $\F_2^n$, denoted by $d_H(\bm x,\bm y)$, is defined as $$d_H(\bm x,\bm y) = w_H(\bm x-\bm y).$$
The {\it minimum distance} of a code $C$, denoted by $d(C)$, is defined to be the minimum distance between distinct codewords in $C$. We write the standard parameters $[n,k,d]$ to describe a code $C$ where $n$ denotes the length of $C$, $k$ its dimension, and $d$ its minimum distance. 

\begin{definition}
Let $C$ be a binary linear code of length $n$. The
{\it dual} of $C$, denoted by $C^{\perp}$, is given by
$$C^{\perp}:= \large\{ \bm y \in \F_2^n \large \;|\; \langle \bm y,\bm x \rangle = 0 \:\: \forall\:
\bm x \in C \large \}.$$
\end{definition}

 Note that, if $C$ is a linear $[n,k]$-code, then $C^{\perp}$ is a linear $[n,n-k]$-code.

\begin{definition}
A binary linear code $C$ is {\it self-orthogonal} if $C \subseteq C^{\perp}$ and  {\it self-dual} if $C = C^{\perp}.$
\end{definition}

\begin{definition}
Let $C$ be a self-dual binary code. If the Hamming weights of all the codewords in $C$ are divisible by $4$, then $C$ is called {\bf Type II} (or doubly-even), otherwise it is called {\bf Type I} (or singly even).
\end{definition}

The following theorem gives an upper bound for the minimum distance of self-dual codes:

\begin{theorem}$($\cite{Rains}$)$\label{extremal} Let $d_I(n)$ and $d_{II}(n)$ be the minimum distance of a Type I and Type II binary code of length $n$. Then
$$d_{II}(n) \leq 4 \lfloor\frac{n}{24}\rfloor+4$$ and
$$d_{I}(n) \leq \left \{
\begin{array}{ll}
 4 \lfloor\frac{n}{24}\rfloor+4 & \textrm{if $n \not \equiv 22 \pmod{24}$}
\\
4 \lfloor\frac{n}{24}\rfloor+6 & \textrm {if $n \equiv 22 \pmod{24}$.}
\end{array}\right.$$
\end{theorem}

Self-dual codes that attain the bounds given in the previous theorem are called {\it extremal}.

\begin{definition}
A binary code $C$ is said to be {\it isodual} if it is permutation equivalent to its dual. 
\end{definition}

\begin{theorem}\label{mindist_def}
If C is generated by $G=[I_k|A]$, then the generator matrix of $C^{\perp}$ is given by $H=[-A^T|I_k]$. 
\end{theorem}
$H$ is also called the parity-check matrix of $C$, namely $C$ is given by 
$$C = \{\bm c \in \F_2^n | H\bm c^{T} = 0\}.$$

There is a natural connection between the parity check matrix of a linear code and the minimum distance which is given by the following theorem:

\begin{theorem}Let $C$ be a linear code and $H$ a parity check matrix for $C$. Then
\par{\bf (i)} $d(C)\geq d$ if and only if any $d-1$ columns of $H$ are linearly independent.
\par{\bf (ii)} $d(C)\leq d$ if and only if $H$ has $d$ columns that are linearly dependent.
\end{theorem}
\begin{corollary}
If $C$ is a linear code and $H$ is a parity check matrix for $C$, then $C$ has minimum distance $d$ if and only if any $d-1$ columns of $H$ are linearly independent and some $d$ columns of $H$ are linearly dependent.
\end{corollary}

The connection of codes and graphs has been explored from different aspects in the literature. The main theme in these works is to take a special class of graphs and construct codes from the adjacency matrix of the graph. The structure of the graph may lead to different types of codes as a result, such as self-dual codes, self-orthogonal codes, etc. We refer the reader to \cite{crnkovic1}--\cite{Oral}, \cite{tonchev} and \cite{tonchev2} for some of these works.

In this work we focus on the following type of construction which was discussed in \cite{tonchev2}. Let $A$ be the adjacency matrix of a simple graph on $n$ vertices. We construct the binary code $C$ from the generator matrix $[I_n|A]$. Such a construction has several advantages which we can describe as follows:
\begin{enumerate}
    \item The dimension of the codes is automatically determined as $n$. So we are looking at $[2n,n]$-codes.
    \item The parity-check matrix of such a code is given by $[A^T|I_n] = [A|I_n]$ since $A$ is symmetric
    \item Note that a permutation of columns of $[A|I_n]$ will bring the matrix into $[I_n|A]$, which means any such code $C$ is isodual. 
    \item Since the codes are isodual, to determine the conditions for self-duality, we just need to factor in the orthogonality conditions. 
\end{enumerate}

In addition to the conditions on the graph that would ensure self-duality of the constructed code, we also find upper bounds on the minimum distances of codes obtained via this construction. We give a combinatorial description to the exact minimum distances of such codes as well as codes obtained from just the adjacency matrix $A$ as the generator matrix. We give examples of isodual and self-dual codes obtained through this construction. 

The rest of the work is organized as follows. In section 2, we give the upper bounds and combinatorial descriptions for the minimum distances of codes obtained from graphs. In section 3, we give necessary and sufficient conditions for the codes to be self-dual, Type I and Type II. In addition, we describe how the join operation on graphs affects the self-duality conditions. We give several examples of self-dual codes obtained through the construction.

\section{The $[I_n|A]$ construction for codes}
As was mentioned in the Introduction, we focus mainly on the construction of binary codes generated by $[I_n|A]$ where $A$ is the adjacency matrix of a simple graph. We start with the following observation:
\begin{obs}
Linear codes generated by $[I_n|A]$ and $[I_n|PAP^T]$, where $P$ is a permutation matrix, are not necessarily the same. For example, consider the graph $P_3$ with adjacency matrix $A$ and permutation matrix $P$ generated by $(1,2)$: 
\begin{align*}
[I_3|A]&=\left[ \begin{array}{ccc|ccc}
1 & 0 & 0 & 0 & 1 & 0 \\
0 & 1 & 0 & 1 & 0 & 1 \\
0 & 0 & 1 & 0 & 1 & 0
\end{array} \right]  \\ 
[I_3|PAP^T]&=\left[ \begin{array}{ccc|ccc}
1 & 0 & 0 & 0 & 1 & 1 \\
0 & 1 & 0 & 1 & 0 & 0 \\
0 & 0 & 1 & 1 & 0 & 0
\end{array} \right]
\end{align*}

\begin{align*}
(1,0,0,0,1,0) \in C([I_3|A]) &=\{(a,b,c,b,a+c,b) \;|\; a,b,c\in \F_2 \} \\ 
(1,0,0,0,1,0) \notin C([I_3|PAP^T]) &=\{(a,b,c,b+c,a,a) \;|\; a,b,c\in \F_2 \}.   
\end{align*}
It is not obvious that $C([I_3|A])$ and  $C([I_3|PAP^T])$ have the same minimum distance.
\end{obs}

The following result gives an upper bound of the minimum distance of a binary code in terms of the $2$-rank (see \cite{Abiad}) of the corresponding graph.
\begin{theorem}\label{rk2}
Let $A$ be the adjacency matrix of a graph on $n$ vertices and let $C$ be the binary linear code generated by $[I_n|A]$. Then we have
$$d(C) \leq rk_2(A)+1,$$
where $rk_2(A)$ denotes the rank of $A$ as a matrix over $\F_2$. 
\end{theorem}

\begin{proof}
Let $G=[I_n|A]$ be the generator matrix of $C$. Then $H = [-A^T|I_n] = [A|I_n]$ is the parity-check matrix of $C$. By the definition of the rank, any set of $rk_2(A)+1$ columns of $A$ are linearly dependent. By Theorem \ref{mindist_def}, $d(C) \leq rk_2(A)+1$.    
\end{proof}

To get a combinatorial interpretation of the minimum distance of $C([I_n|A])$, we study the following set of vertices of a graph $\Gamma$ with adjacency matrix $A$ and vertex set $V$: for a nonempty subset $S$ of $V$, the set of vertices of $\Gamma$ with odd number of neighbors in $S$ is denoted by $\von(S)$, i.e.,
$$\von(S)=\{v\in V\;:\: |\n(v)\cap S| \text{ is odd}\},$$
where $\n(v)$ denotes the set of neighbors of the vertex $v$ in $\Gamma$. Note that $S$ and $\von(S)$ have no inclusion-exclusion relationship that holds for all graphs as evident in the following examples.

\begin{example}\blank
\begin{enumerate}
\item Consider $\Gamma=C_4$ with vertices $1,2,3,4$ consecutively adjacent. For $S=\{1\}$, $\von(S)=\{2,4\}=\n(1)$. For $S=\{1,2\}$, $\von(S)=\{1,2,3,4\}$.  For $S=\{1,3\}$, $\von(S)=\varnothing$. 

    \item Consider $\Gamma=K_n,\; n\geq 4$ with vertex set $\{1,2,\ldots,n\}$. For $S=\{1\}$, $\von(S)=\{2,3,\ldots,n\}=\n(1)$. For $S=\{1,2\}$, $\von(S)=\{1,2\}$.   For $S=\{1,2,\ldots,n-1\}$ with even $n$, $\von(S)=\{n\}$. For $S=\{1,2,\ldots,n-1\}$ with odd $n$, $\von(S)=\{1,2,\ldots,n-1\}$. 
    
\end{enumerate}

\end{example}

\begin{definition}
Two vertices $u$ and $v$ of a graph are called {\it duplicate vertices} if they are not adjacent and $\n(u)=\n(v)$, i.e., they have the same neighbors.
\end{definition}

\begin{obs}
Let $\Gamma$ be a graph. If $S$ is a set of two duplicate vertices of $\Gamma$, then $\von(S)=\varnothing$.
\end{obs}
\begin{proof}
Let $A$ be the adjacency matrix of $\Gamma$ and $A_i$ denote the column $i$ of $A$. Without loss of generality let $S=\{1,2\}$. If $1$ and $2$ are duplicate vertices, then $A_1+A_2 \equiv 0 \Mod{2}$ which implies $\von(S)=\varnothing$.
\end{proof}

Linear dependency relations among columns of a matrix associated with graphs have been studied in \cite{minskewrank4}. We study the same in connection with $\von$.
\begin{theorem}
Let $G$ be a graph with vertex set $V$ and adjacency matrix $A$. Let $S$ be a nonempty subset of $V$. If  $\von(S)=\varnothing$, then the columns of $A$ corresponding to $S$ are linearly dependent. Conversely, if the columns of $A$ corresponding to $S$ are minimally linearly dependent, then $\von(S)=\varnothing$.
\end{theorem}
\begin{proof}
Suppose $S=\{i_1,i_2,\ldots,i_k\}$ and $\von(S)=\varnothing$. Then 
$$A_{i_1}+A_{i_2}+\cdots+A_{i_k} \equiv 0 \Mod{2}$$
which implies columns $A_{i_1},A_{i_2},\ldots,A_{i_k}$ of $A$ are linearly dependent.

Conversely, suppose $S=\{1,2,\ldots,k\}$ and $A_1,A_2,\ldots,A_k$ are minimally linearly dependent. Then  $A_1+A_2+\cdots+A_k \equiv 0 \Mod{2}$. If $i\in \von(S)$, then 
$$(A_1+A_2+\cdots+A_k)_i \equiv 1 \Mod{2},$$
a contradiction. Thus $\von(S)=\varnothing$.
\end{proof}

\begin{corollary}
Let $\Gamma$ be a graph with vertex set $V$ and adjacency matrix $A$. Let $S$ be a nonempty subset of $V$. The columns of $A$ corresponding to $S$ are minimally linearly dependent if and only if one of the following is true:
\begin{enumerate}
    \item[(a)] $S$ consists of a single isolated vertex of $\Gamma$.
    
    \item[(b)] $\von(S\setminus \{i\})=\n(i)$ for each $i\in S$ and there is no proper subset $S'$ of $S$ for which $\von(S'\setminus \{i\})=\n(i)$ for each $i\in S'$.
\end{enumerate}
\end{corollary}

Now we discuss linear dependence among columns of $[A|I_n]$ where $A$ is the adjacency matrix of a graph $\Gamma$ on $n$ vertices.
\begin{theorem}\label{mindist1}
Let $\Gamma$ be a graph on $n$ vertices with vertex set $V$ and adjacency matrix $A$. If $S$ is a nonempty subset of $V$, then the columns of $A$ indexed by $S$ and the columns of $I_n$ indexed by $\von(S)$ are $(|S|+|\von(S)|)$ linearly dependent columns of $[A|I_n]$. Conversely, if the set of columns of $[A|I_n]$ indexed by the set $S'\subseteq \{1,2,\ldots, 2n\}$ is minimally linearly dependent, then it is the union of the columns of $A$ indexed by $S$ and the columns of  $I_n$  indexed by $\von(S)$ for some nonempty subset $S$ of $V$, in other words $S'=S\cup \{n+i \;|\; i\in \von(S) \}$.
\end{theorem}
\begin{proof}
Let $\varnothing \neq S\subseteq V$. 
Let $c$ be the sum of columns of $A$ indexed by $S$. Then $c_i$, the $i$th entry of $c$, is the number of vertices of $S$ adjacent to vertex $i$. Therefore if  vertex $i$ is adjacent to an even number of vertices in $S$, then $c_i\equiv 0 \Mod{2}$. Similarly if  vertex $i$ is adjacent to an odd number of vertices in $S$, then $c_i\equiv 1 \Mod{2}$. Thus the only entries of $c$ that are $1 \Mod{2}$ correspond to  $\von(S)$. So if we add $c$ with the columns of $I_n$ with indices corresponding to $\von(S)$, the sum would be a zero vector. 

Conversely, suppose the set of $d$ columns of $[A|I_n]$ indexed by the set $S'\subseteq \{1,2,\ldots, 2n\}$ is minimally linearly dependent. Without loss of generality suppose $S'$ is the union of $S=\{1,2,\ldots,k\}$, $k\leq n$ and $T\subseteq \{n+1,n+2,\ldots, 2n\}$.

Case 1. $T=\varnothing$ (i.e., $S'=S$)\\
Since $A_1,A_2,\ldots,A_k$ are minimally linearly dependent, $A_1+A_2+\cdots+A_k \equiv 0 \Mod{2}$. It suffices to show that $\von(S)=\varnothing$. If not, let $i\in \von(S)$. Then 
$$(A_1+A_2+\cdots+A_k)_i \equiv 1 \Mod{2},$$
a contradiction.

Case 2. $T\neq \varnothing$\\
Let $T=\{n+i_1,n+i_2,\ldots,n+i_{d-k}\}$ and $e_j$ be  column $j$ of $I_n$ for $j=i_1,i_2,\ldots,i_{d-k}$. Since $A_1,A_2,\ldots,A_k,e_{i_1},e_{i_2},\ldots,e_{i_{d-k}}$ are minimally linearly dependent, $$A_1+A_2+\cdots+A_k+e_{i_2}+\cdots+e_{i_{d-k}} \equiv 0 \Mod{2}.$$
Then 
$$A_1+A_2+\cdots+A_k \equiv e_{i_1}+e_{i_2}+\cdots+e_{i_{d-k}} \Mod{2}$$ 
which implies $\von(S)=\{ i_1,i_2,\ldots,i_{d-k}\}$ because $e_{i_1},e_{i_2},\ldots,e_{i_{d-k}}$ are columns of $I_n$. Thus $S'=S\cup \{n+i \;|\; i\in \von(S) \}$.
\end{proof}

As a consequence of the preceding theorem, we have the following result.

\begin{theorem}\label{mindist2}
Let $A$ be the adjacency matrix of a graph $\Gamma$  on $n$ vertices with vertex set $V$. Let $C$ be the binary linear code generated by $[I_n|A]$. Then the minimum distance $d(C)$ of $C$ is given by
$$d(C) = \min_{\varnothing \neq S\subseteq V} (|S|+|\von(S)|).$$
\end{theorem}
\begin{proof}
First note that $H=[A|I_n]$ is the parity-check matrix of $C$.  By Theorem \ref{mindist1}, a code word in $C$ with weight $d(C)$ corresponds to minimally dependent columns of $H=[A|I_n]$ indexed by $S\cup \{n+i \;|\; i\in \von(S) \}$ for some nonempty subset $S$ of $V$. Then 
$$d(C) \geq  \min_{\varnothing \neq S\subseteq V} (|S|+|\von(S)|).$$
If there is a nonempty subset $S$ of $V$  for which $d(C)> |S|+|\von(S)|$, then by Theorem \ref{mindist1} we find $(|S|+|\von(S)|)$ linearly dependent columns of $H=[A|I_n]$ giving a codeword of $C$ with weight less than $d(C)$, a contradiction. Thus the equality holds.
\end{proof}

\begin{corollary}
Let $A$ be the adjacency matrix of a graph $\Gamma$  on $n$ vertices. Let $P$ be an $n\times n$ permutation matrix. 
Then the binary linear codes generated by $[I_n|A]$ and $[I_n|PAP^T]$ are not necessarily the same but they have the same minimum distance.
\end{corollary}
\begin{proof}
The graph with adjacency matrix $PAP^T$ is isomorphic to $\Gamma$. Then the binary linear codes generated by $[I_n|A]$ and $[I_n|PAP^T]$ have the same minimum distance by Theorem \ref{mindist2}.
\end{proof}

By Theorem \ref{rk2} and Theorem \ref{mindist2}, we have the following lower bound of the $2$-rank of a graph:
\begin{corollary}
Let $A$ be the adjacency matrix of a graph $\Gamma$  with vertex set $V$. Then $$-1+\min_{\varnothing \neq S\subseteq V} (|S|+|\von(S)|) \leq rk_2(A).$$
\end{corollary}

\begin{question}
Characterize the graphs $\Gamma$ with the adjacency matrix $A$ and  the vertex set $V$ for which
$$rk_2(A)=-1+\min_{\varnothing \neq S\subseteq V} (|S|+|\von(S)|).$$
\end{question}

\begin{example}
The following are examples of binary linear code $C$ generated by $[I_n|A]$ where $A$ is the adjacency matrix of a graph $\Gamma$  on $n$ vertices.
\begin{enumerate}
    \item For $\Gamma=P_n,\; n\geq 2$, $d(C) =2= |S|+|\von(S)|$ where $S=\{1\}$ and $\von(S)=\{2\}$.\\
    For $\Gamma=P_1$, $d(C) =1= |S|+|\von(S)|$ where $S=\{1\}$ and $\von(S)=\varnothing$.
    
    \item When $\Gamma$ is a tree on $n\geq 2$ vertices, $d(C) =2= |S|+|\von(S)|$ where $S=\{v\}$ consisting of a pendant vertex $v$ and $\von(S)=\{w\}$ where $w$ is adjacent to $v$.
    
    \item For $\Gamma=C_n,\; n\geq 5$, $d(C) =3= |S|+|\von(S)|$ where $S=\{1\}$ and $\von(S)=\{2,n\}$.\\
    For $\Gamma=C_4$, $d(C) =2= |S|+|\von(S)|$ where $S=\{1,3\}$ and $\von(S)=\varnothing$.\\
    For $\Gamma=C_3$, $d(C) =3= |S|+|\von(S)|$ where $S=\{1\}$ and $\von(S)=\{2,3\}$.
    
    \item For $\Gamma=K_{n},\; n\geq 4$, $d(C) =4= |S|+|\von(S)|$ where $S=\{1,2\}$ and $\von(S)=\{1,2\}$.
    
    \item For star $\Gamma=K_{1,n},\; n\geq 1$ centered at $1$, $d(C) =2= |S|+|\von(S)|$ where $S=\{2\}$ and $\von(S)=\{1\}$.
    
    \item For  $\Gamma=K_{m,n},\; m \text{ or } n\geq 2$, $d(C) =2= |S|+|\von(S)|$ where $S=\{1,2\}$ and $\von(S)=\varnothing$ because of duplicate vertices 1 and 2. Recall $C_4=K_{2,2}$.
    
    \item For $G=W_n=K_1 \vee C_{n-1},\; n\geq 6$ centered at $1$, $d(C) =4= |S|+|\von(S)|$ where $S=\{2\}$ and $\von(S)=\{1,3,n\}$.\\
    For $\Gamma=W_5$, $d(C) =2= |S|+|\von(S)|$ where $S=\{2,4\}$ and $\von(S)=\varnothing$.\\
    For $\Gamma=W_4$, $d(C) =4= |S|+|\von(S)|$ where $S=\{1\}$ and $\von(S)=\{2,3,4\}$.
    
    \item When $\Gamma$ is the Petersen graph which is the $srg(10, 3, 0, 1)$,  $d(C) =4= |S|+|\von(S)|$ where $S=\{1\}$ and $\von(S)=\n(1)=\{2,5,6\}$ where the outer vertices are 1,2,3,4 in the standard drawing.
    
\end{enumerate}
\end{example}

\begin{remark}
Suppose $V$ is the vertex set of $K_n$ and let $S\subseteq V$. It is easy to observe that 
\begin{equation*}
\von(S) =  \left\{
\begin{array}{ll}
S & \text{if $|S|$ is even} \\
V\setminus S & \text{if $|S|$ is odd.}%
\end{array}%
\right.
\end{equation*}
\end{remark}

\begin{obs}
Let $A$ be the adjacency matrix of a graph $\Gamma$  on $n$ vertices with vertex set $V$. Let $C$ be the binary linear code generated by $[I_n|A]$ and  $S$ be a nonempty subset of $V$  for which $d(C) = |S|+|\von(S)|$. From the preceding remark for $K_n$, we have either $S=\von(S)$ or $S\cap \von(S)=\varnothing$. At least one of these two properties seems to hold for other graphs also.
\end{obs}

\begin{conjecture}
Let $A$ be the adjacency matrix of a graph $\Gamma$  on $n$ vertices with vertex set $V$. Let $C$ be the binary linear code generated by $[I_n|A]$. Suppose  $S$ is a nonempty subset of $V$  for which $d(C) = |S|+|\von(S)|$. Then either $S=\von(S)$ or $S\cap \von(S)=\varnothing$.
\end{conjecture}

The following observation may be helpful for the future work on the preceding conjecture.
\begin{obs}
If $v\in S\cap \von(S)$ and $|S|\geq 2$, then 
$$\von(S\setminus\{v\}) \setminus \n(v)=\von(S) \setminus \n(v) \text{ and } 
\von(S\setminus\{v\}) \cap \von(S) \cap \n(v) =\varnothing.$$
\end{obs}

\section{Self-dual codes from graphs}
We start by observing that if $A$ is an $n\times n$ matrix, then the binary code generated by $[I_n|A]$ is a self-dual code if and only if $AA^T=I_n$, where the matrix multiplication is done in $\F_2$. 
If $A$ is the adjacency matrix of a simple graph, then this condition is reduced to $A^2 \equiv I_n \pmod{2}$. We first give some results on the graphs for which this condition is satisfied.

\begin{theorem}\label{main}
Let $\Gamma$ be a graph on $n$ vertices $1,2,\ldots,n$ with adjacency matrix $A$. Then $A^2\equiv I_n \Mod{2}$ if and only if the following are true:
\begin{enumerate}
\item[(a)]  $\deg (i)$ is odd for all vertices $i=1,2,\ldots,n$ (This implies $n$ is even).

\item[(b)] $|N(i)\cap N(j)|$ is even for all vertices $i\neq j$.
\end{enumerate}
\end{theorem}

\begin{proof}
Suppose $A^2\equiv I_n \Mod{2}$. Then for all $i=1,2,\ldots,n$, $\sum_{k=1}^n a_{ik}^2 \equiv 1 \Mod{2}$, i.e., $\sum_{k=1}^n a_{ik}^2$ is odd and consequently $\deg (i)$ is odd.  Since every graph has an even number of odd-degree vertices, $n$ is even by (a). Since $A^2\equiv I_n \Mod{2}$, for all vertices $i\neq j$, $\sum_{k=1}^n a_{ik}a_{jk}$ is even and consequently $|N(i)\cap N(j)|$ is even. The converse follows by  similar arguments.
\end{proof}

We prove the following lemma which will be used in the proof of the subsequent theorem:
\begin{lemma}\label{LemmaTypeII}
Let $C$ be a self-orthogonal code and assume $\bm c_1, \bm c_2$ are two codewords with 
$$w(\bm c_1) \equiv w(\bm c_2) \equiv 0 \pmod{4}.$$
Then $w(\bm c_1+\bm c_2) \equiv 0 \pmod{4}$. 
\end{lemma}

\begin{proof}
Since $C$ is self-orthogonal, we have $\langle \bm c_1, \bm c_2\rangle =0$, which implies the weight $w(\bm c_1\circ \bm c_2)$ of the entry-wise product $\bm c_1\circ \bm c_2$ of $\bm c_1$ and $\bm c_2$ is even. Thus we have 
$$w(\bm c_1+\bm c_2) = w(\bm c_1)+w(\bm c_2) - 2w(\bm c_1\circ \bm c_2) \equiv 0\pmod{4}.$$
\end{proof}

We are now ready to prove the following theorem which gives a necessary and sufficient condition for a graph to generate a Type II code:

\begin{theorem}\label{TypeII}
Let $A$ be the adjacency matrix of a simple graph on $n$ vertices, which satisfies the hypothesis of Theorem \ref{main}. Then $[I_n|A]$ generates a Type II code if and only if $\deg(v) \equiv 3  \pmod{4}$ for all vertices $v$ of the graph.  
\end{theorem}
\begin{proof}
The necessity being clear, we proceed to proving the sufficiency. 

Suppose $\deg(v)\equiv 3 \pmod{4}$ for all vertices $v$. Then each row of $G = [I_n|A]$ has weight divisible by $4$. Suppose the rows of $G$ are denoted by $\bm r_1, \bm r_2, \ldots, \bm r_n$. So we have $w(\bm r_i) \equiv 0 \pmod{4}$ for $i=1,2,\ldots, n$. Then we claim that all the codewords will have weight divisible by $4$. Note that every codeword of $C([I_n|A])$ is obtained from a sum of the form 
$\bm r_{i_1}+\bm r_{i_2}+\dots +\bm r_{i_k}$, where 
$1\leq i_1<i_2<\dots < i_k\leq n$ and $k\geq 1$. We proceed by induction on $k$.

If $k=1$, then we have just the rows $\bm r_i$, which all have weights divisible by 4. 

Assume the assertion to be true for all sums with $k\geq 1$ summands. 

Let 
$$\bm c=\bm r_{i_1}+\bm r_{i_2}+\dots +\bm r_{i_k}+\bm r_{i_{k+1}} = \left (\bm r_{i_1}+\bm r_{i_2}+\dots +\bm r_{i_k}\right) + \bm r_{i_{k+1}}.$$
By induction hypothesis, $\left (\bm r_{i_1}+\bm r_{i_2}+\dots +\bm r_{i_k}\right)$ has weight divisible by $4$. Note that both $\left (\bm r_{i_1}+\bm r_{i_2}+\dots +\bm r_{i_k}\right)$ and $\bm r_{i_{k+1}}$ are codewords in $C$, which is self-dual and both codewords have weights divisible by $4$. So by Lemma \ref{LemmaTypeII}, the weight of $c$ is divisible by $4$. 
\end{proof}

Since Type II binary self-dual codes only exist for lengths that are multiple of $8$, we have the following combinatorial result as a consequence of the previous theorem:
\begin{corollary}
Let $n\equiv 2 \pmod{4}$. Then a simple graph on $n$ vertices that satisfies the hypotheses of Theorem \ref{main}  has at least one vertex whose degree is $1 \pmod{4}$.  
\end{corollary}

In the following theorem, we explore the special case of complete graphs. 
\begin{theorem}
Let $C$ be the code generated by $[I_n|A]$, where $A$ is the adjacency matrix of $K_n$. $C$ is a self-dual code if and only if $n$ is even.  Moreover, if $n\geq 4$ we have
\par {\bf a)} $C$ is a Type II self-dual code of parameters $[2n,n,4]$ if $n$ is divisible by $4$.
\par {\bf b)} $C$ is a Type I self-dual code of parameters $[2n,n,4]$ if $n$ is not divisible by $4$. 
\end{theorem}

\begin{proof}
By Theorem \ref{main}, $C$ is self-dual if and only if $n$ is even. 

\par{\bf a)} If $n=4k$, then the degree of every vertex of $K_n$ is $4k-1$, which, by Theorem \ref{TypeII}, implies that the code generated by $[I_n|A]$ is Type II. This means $d(C)\geq 4$. But the sum of any two rows of $[I_n|A]$ has weight 4, which means $d(C)=4$. 

\par {\bf b)} If $n=4k+2$ with $k\geq 1$, then every row of $[I_n|A]$ has weight $4k+2$, which makes $C$ Type I. To find the minimum distance, we use Theorem \ref{mindist2}. Let $S$ be a nonempty subset of the vertices of $K_n$. If $S = \{x\}$, then $|\von(S)| = n-1$, which means $|S|+|\von(S)| = n \geq 4$. 

If $S = \{x,y\}$, then $\von(S)=S$, which means $|S|+|\von(S)| = 4$.

If $S = \{x, y, z\}$, then $|\von(S)| = n-3$, which means $|S|+|\von(S)| = n \geq 4$. 

If $|S| \geq 4$, then $|S|+|\von(S)| \geq 4$. 
Thus the minimum distance is 4.
\end{proof}

\begin{corollary}
The code generated by $K_n$ is an extremal Type II self-dual code for $n=4$ and $n=8$. The code generated by $K_n$ is an extremal Type I self-dual code for $n=6$.  
\end{corollary}

\begin{example}
Consider the regular graph $\Gamma=2K_4$  which is the $srg(8, 3, 2, 0)$  with the following adjacency matrix $A$. Since $A^2\equiv I_8 \Mod{2}$ and each vertex of $\Gamma$ has degree $3$, the binary code $C=C([I_8|A])$  is an extremal Type II self-dual $[16,8,4]$ code by Theorems \ref{main}, \ref{TypeII}, and \ref{extremal}. 
$$A=\left[ \begin{array}{cccccccc}
0& 0& 0& 1& 0& 1& 0& 1\\
0& 0& 1& 0& 1& 0& 1& 0\\
0& 1& 0& 0& 1& 0& 1& 0\\
1& 0& 0& 0& 0& 1& 0& 1\\
0& 1& 1& 0& 0& 0& 1& 0\\
1& 0& 0& 1& 0& 0& 0& 1\\
0& 1& 1& 0& 1& 0& 0& 0\\
1& 0& 0& 1& 0& 1& 0& 0 
\end{array}\right]$$

\end{example}

\begin{theorem}
Let $\Gamma$ be a strongly regular graph with parameters $(n,k,\lambda,\mu)$ and adjacency matrix $A$. Suppose $C$ is a linear code generated by $[I_n|A]$. Then $C$ is self-dual if and only if $k$ is odd and $n,\lambda,\mu$ are even.
\end{theorem}
\begin{proof}
For a strongly regular graph with parameters $(n,k,\lambda,\mu)$, $\deg(i) =k$ for all $i=1,2,\ldots,n$ and $|N(i)\cap N(j)|$ is $\lambda$ or $\mu$ for all $i\neq j$. Thus $A^2\equiv I_n \Mod{2}$ if and only if $k$ is odd and $\lambda,\mu$ are even by Theorem \ref{main}. 
\end{proof}

\begin{question}
Let $\Gamma$ be a strongly regular graph with parameters $(n,k,\lambda,\mu)$ and adjacency matrix $A$. Find the minimum distance of the linear code $C([I_n|A])$ in terms of $n,k,\lambda,\mu$.
\end{question}

Now we explore effects of graph operations on corresponding linear codes. In particular, we study the join $\Gamma_1\vee \Gamma_2$ of two graphs $\Gamma_1$ and $\Gamma_2$ with disjoint vertex sets $V_1$ and $V_2$ respectively. Note that $\Gamma_1\vee \Gamma_2$ has the vertex set $V_1\cup V_2$ and the edge set consisting of all edges of $\Gamma_1$ and $\Gamma_2$ together with all edges between them. For example, $K_{m,n}$ is the join of $mK_1$ and $nK_1$.

\begin{theorem}
Let $\Gamma_1$ and $\Gamma_2$ be two graphs on $n_1$ and $n_2$ vertices with adjacency matrices $A_1$ and $A_2$ respectively. Suppose $\Gamma_1\vee \Gamma_2$ is the join of $\Gamma_1$ and $\Gamma_2$ with adjacency matrix $A$. If $C([I_{n_1}|A_1])$ and $C([I_{n_2}|A_2])$ are self-dual codes, then so is $C([I_{n_1+n_2}|A])$. 
\end{theorem}

\begin{proof}
First note that
$$A=\left[ \begin{array}{cc}
A_1& J_{n_1,n_2}\\
J_{n_2,n_1}& A_2
\end{array}\right]
\text{ and }
A^2=\left[ \begin{array}{cc}
n_2J_{n_1,n_1}+A_1^2 & A_1 J_{n_1,n_2} +J_{n_1,n_2}A_2 \\
J_{n_2,n_1}A_1+A_2 J_{n_2,n_1}& n_1J_{n_2,n_2}+A_2^2
\end{array}\right].$$

Suppose $C([I_{n_1}|A_1])$ and $C([I_{n_2}|A_2])$ are self-dual codes. By Theorem \ref{main},  $A_1^2\equiv I_{n_1} \Mod{2}$ and $A_2^2\equiv I_{n_2} \Mod{2}$ which imply $n_1\equiv n_2\equiv 0 \pmod{2}$ and degree of each vertex in $\Gamma_1$ and $\Gamma_2$ is  $1 \pmod{2}$. Then 
$$A_1 J_{n_1,n_2} +J_{n_1,n_2}A_2 \equiv J_{n_1,n_2} \pmod{2}+J_{n_1,n_2} \pmod{2} \equiv O_{n_1,n_2} \pmod{2},$$
$$J_{n_2,n_1}A_1+A_2 J_{n_2,n_1}\equiv J_{n_2,n_1} \pmod{2}+J_{n_2,n_1}\pmod{2} \equiv O_{n_2,n_1} \pmod{2}.$$

$$A^2=\left[ \begin{array}{cc}
n_2J_{n_1,n_1}+A_1^2 & A_1 J_{n_1,n_2} +J_{n_1,n_2}A_2 \\
J_{n_2,n_1}A_1+A_2 J_{n_2,n_1}& n_1J_{n_2,n_2}+A_2^2
\end{array}\right]\equiv I_{n_1+n_2}\pmod{2}.$$

Thus $C([I_{n_1+n_2}|A])$ is a self-dual code by Theorem \ref{main}.
\end{proof}

The following theorem describes the type of the join of self-dual codes. 
\begin{theorem}
Let $\Gamma_1$ and $\Gamma_2$ be two graphs on $n_1$ and $n_2$ vertices and with generator matrices $A_1$ and $A_2$ respectively. Suppose that $A$ is the generator matrix of $\Gamma_1 \vee \Gamma_2$.  
\begin{enumerate}
    \item[(a)] When $n_1\equiv n_2 \equiv 0 \pmod{4}$,   $C([I_{n_1+n_2}|A])$ is Type II if and only if both $C([I_{n_1}|A_1])$ and $C([I_{n_2}|A_2])$ are Type II.
    
    \item[(b)] When $n_1\equiv n_2 \equiv 2 \pmod{4}$, if $C([I_{n_1+n_2}|A])$ is Type II, then both $C([I_{n_1}|A_1])$ and $C([I_{n_2}|A_2])$ are Type I and the converse is true if each vertex of $\Gamma_1$ and $\Gamma_2$ has degree $1 \pmod{4}$.
    
    \item[(c)]If exactly one of $n_1$ and $n_2$ is divisible by $4$, then $C([I_{n_1+n_2}|A])$ is Type I.
\end{enumerate}
\end{theorem}

\begin{proof}
 First we observe that for any vertex $v$ in $\Gamma_1$, the degree of $v$ in $\Gamma_1\vee \Gamma_2$ is $n_2+deg(v)$. Similarly for any vertex $w$ in $\Gamma_2$, the degree of $w$ in $\Gamma_1\vee \Gamma_2$  is $n_1+deg(w)$. Then the cases $n_1\equiv n_2 \equiv 0 \pmod{4}$ and  $n_1\equiv n_2 \equiv 2 \pmod{4}$ follow from Theorem \ref{TypeII}.

Now consider the case when exactly one of $n_1$ and $n_2$ is divisible by $4$. Then we have $n_1+n_2\equiv 2 \pmod{4}$, which implies $2(n_1+n_2) \equiv 4 \pmod{8}$. Since Type II codes only exist for lengths that are multiples of 8,  $C([I_{n_1+n_2}|A])$ is not Type II, hence Type I. 
\end{proof}

We end by the following results about the minimum distance of $C([I_{n_1+n_2}|A])$ and its connection to the minimum distances of $C([I_{n_1}|A_1])$ and $C([I_{n_2}|A_2])$.

\begin{theorem}
Let $\Gamma_1$ and $\Gamma_2$ be two graphs with disjoint vertex sets $V_1$ and $V_2$ of sizes $n_1$ and $n_2$ respectively. Let $A_1$, $A_2$, and $A$ be the adjacency matrices of $\Gamma_1$, $\Gamma_2$, and $\Gamma_1\vee \Gamma_2$ respectively. Suppose that $d_1, d_2$, and $d$ are the minimum distances of the codes generated by $[I_{n_1}|A_1]$, $[I_{n_2}|A_2]$, and $[I_{n_1+n_2}|A]$ respectively. 
\begin{enumerate}
    \item[(a)] Suppose  $S_i$ is a nonempty subset of $V_i$  for which $d_i= |S_i|+|\von(S_i)|$ for $i=1,2$. If $|S_i|$ is even for some $i=1,2$, then
$$d\leq d_i.$$
    
    If $|S_1|$ and $|S_2|$ are odd, then 
$$d\leq
\min \{n_2+d_1,n_1+d_2 \}.$$

    \item[(b)] Suppose $S=S_1\cup S_2$ is a nonempty subset of $V_1\cup V_2$ for which $d = |S|+|\von(S)|$ where $\varnothing \neq S_1\subseteq V_1$ and $\varnothing \neq S_2\subseteq V_2$. If at least one of $|S_1|$ and $|S_2|$ is even, then
    $$d_1+ d_2 \leq d.$$
\end{enumerate}

\end{theorem}

\begin{proof}
(a) We prove this by the following cases:

\noindent Case 1. $|S_1|$ is even.\\
For $S=S_1 \subseteq V_1\cup V_2$ in $\Gamma_1\vee \Gamma_2$, we have $\von(S)=\von(S_1)$ in $\Gamma_1$. Then by Theorem \ref{mindist2},
$$d\leq |S_1|+|\von(S_1)|=d_1.$$

\noindent Case 2. $|S_2|$ is even.\\
For $S=S_2 \subseteq V_1\cup V_2$ in $\Gamma_1\vee \Gamma_2$, we have $\von(S)=\von(S_2)$ in $\Gamma_2$. Then by Theorem \ref{mindist2},
$$d\leq |S_2|+|\von(S_2)|=d_2.$$

\noindent Case 3. $|S_1|$ and $|S_2|$ are odd.\\
In $\Gamma_1\vee \Gamma_2$, let $S=S_1$. Thus $\von(S)=\von(S_1)\cup V_2$. Then by Theorem \ref{mindist2},
$$d\leq |S_1|+|\von(S_1)|+|V_2|=n_2+d_1.$$
Similarly
$$d\leq |S_2|+|\von(S_2)|+|V_1|=n_1+d_2.$$

\bigskip
(b) We prove this by the following cases:

\noindent Case 1. $|S_1|$ and $|S_2|$ are even.\\
    $\von(S)$ is the union of $\von(S_1)$ in $\Gamma_1$ and $\von(S_2)$ in $\Gamma_2$. Then by Theorem \ref{mindist2},
    $$d_1+d_2\leq |S_1|+|S_2|+|\von(S_1)|+|\von(S_2)|= d.$$
    
\noindent Case 2. $|S_1|$ is odd and $|S_2|$ is even.\\
    $\von(S)$ is the union of $V_2$ and $\von(S_1)$ in $\Gamma_1$. Then by Theorem \ref{mindist2},
$$d_1+ d_2 \leq |S_1|+|S_2|+|\von(S_1)|+|V_2|=d.$$
    
\noindent Case 3. $|S_1|$ is even and $|S_2|$ is odd.\\
    Similar to Case 2, we have 
    $$d_1+ d_2\leq d.$$

\end{proof}

\end{document}